\documentclass[12pt,a4paper]{amsart}
\usepackage{latexsym}
\usepackage{amsmath,amssymb,exercise}
\usepackage{dsfont}
\usepackage{wasysym}

\usepackage[bookmarks=true,hyperindex,pdftex,colorlinks,citecolor=blue]{hyperref}

 \newtheorem{theo}{Theorem}[section]
 
 \newtheorem{lem}[theo]{Lemma}
\newtheorem{propos}[theo]{Proposition}
\newtheorem{prob}[theo]{Problem}
 \newtheorem{corollary}[theo]{Corollary}
 \theoremstyle{definition}
 \newtheorem{definition}[theo]{Definition}

\newcommand{\dopu}{{:}\allowbreak\ }

\newcommand{\R}{{\mathbb{R}}}

\newcommand{\N}{{\mathbb{N}}}

\renewcommand{\leq}{\leqslant}
\renewcommand{\le}{\leqslant}
\renewcommand{\geq}{\geqslant}
\renewcommand{\ge}{\geqslant}

\newcommand{\loglike}[1]{\mathop{\rm #1}\nolimits}
\newcommand{\ex}{\loglike{ext}}

\newcommand{\conv}{{\loglike{conv}}}

\newcommand{\spn}{\loglike{span}}

\newcommand{\bea}{\begin{eqnarray*}}
\newcommand{\eea}{\end{eqnarray*}}
\newcommand{\beq}{\begin{eqnarray}}
\newcommand{\eeq}{\end{eqnarray}}

\def\epsilon{\varepsilon}

\def\cB{{{\mathcal B}}}
\def\cU{{{\mathcal U}}}

\numberwithin{equation}{section}

\begin{document}
\dedicatory{Dedicated to the memory of Bernardo Cascales}
\title[Non-expansive bijections, uniformities and faces]{Non-expansive bijections, uniformities and polyhedral faces}

\author[Angosto]{Carlos Angosto}
\address[Angosto]{Departamento de Matem\'atica Aplicada y
 Estad\'{\i}stica Universidad Polit\'ecnica de Cartagena
\newline \href{https://orcid.org/0000-0001-7592-6121}{ORCID: \texttt{0000-0001-7592-6121}}}
\email{carlos.angosto@upct.es}

\author[Kadets]{Vladimir Kadets}

\address[Kadets]{ School of Mathematics and Informatics, V.N. Karazin Kharkiv National University, 61022 Kharkiv, Ukraine
\newline \href{http://orcid.org/0000-0002-5606-2679}{ORCID: \texttt{0000-0002-5606-2679}}}
\email{v.kateds@karazin.ua }

\author[Zavarzina]{Olesia Zavarzina}
\address[Zavarzina]{ School of Mathematics and Informatics, V.N. Karazin Kharkiv National University, 61022 Kharkiv, Ukraine
\newline \href{http://orcid.org/0000-0002-5731-6343}{ORCID: \texttt{0000-0002-5731-6343}}}
\email{olesia.zavarzina@yahoo.com}


\thanks{The research of the first author is supported by MINECO grant MTM2014-57838-C2-1-P and Fundaci\'on S\'eneca, Regi\'on de Murcia grant 19368/PI/14. The research of the second author is done in frames of Ukrainian Ministry of Science and Education Research Program 0118U002036, and it was partially supported by Spanish MINECO/FEDER projects MTM2015-65020-P and MTM2017-83262-C2-2-P and Fundaci\'on S\'eneca, Regi\'on de Murcia grant 19368/PI/14.
The third author is partially supported by a grant from Akhiezer's Fund, 2018. }

\begin{abstract} We extend the result of B. Cascales at al. about expand-contract plasticity of the unit ball of strictly convex Banach space to those spaces whose unit ball is the union of all its finite-dimensional polyhedral extreme subsets. We also extend the definition of expand-contract plasticity to uniform spaces and generalize the theorem on expand-contract plasticity of totally bounded metric spaces to this new setting.
\end{abstract}

\date{July 26, 2018}
\maketitle

%%%%%%%%%%%%%%%%%%%%%%
%%%% Introduction %%%%%%%
%%%%%%%%%%%%%%%%%%%%%%

\section{Introduction}

Let $E_1, E_2$ be metric spaces. A map $F \colon E_1 \to E_2$ is called non-expansive (resp. non-contractive) if $d(F(x),F(y)) \leq d(x, y)$ (resp. $d(F(x),F(y)) \geq d(x, y)$)
for all $x,y \in E_1$. A metric space $E$ is called \emph{expand-contract plastic} (or simply, an EC-space) if every non-expansive bijection from $E$ onto itself is an isometry.

\cite[Satz IV]{FH1936} or \cite[Theorem 1.1]{NaiPioWing} imply that every totally bounded metric space is an EC-space, but there are examples of EC-spaces that are not totally bounded (and even unbounded).

In general bounded closed subsets of infinite-dimensional Banach spaces are not EC-spaces, see \cite[Example 2.7]{CKOW2016}. It is not known whether it is true that for every Banach space $X$ its unit ball $B_X$ is an EC-space. There is no known counterexample and there are some known partial positive results: finite-dimesional spaces (the unit ball is compact), strictly convex Banach spaces (see \cite[Theorem 2.6]{CKOW2016}) or $\ell_1$-sum of strictly convex Banach spaces (see \cite[Theorem 3.1]{KZ2017}). A more general problem is studied in \cite{Zav}:

\begin{prob}\label{prob:ecbanach}
Let $X$ and $Y$ be Banach spaces and $F:B_X\to B_Y$ be a bijective non-expansive map. Is $F$ an isometry?
\end{prob}

There are positive answers when $Y$ is $\ell_1$, a finite-dimensional Banach space or a strictly convex Banach space (see \cite{Zav}).

\vspace{2mm}
The unit sphere of a strictly convex space consists of its extreme points. The main result of our paper is Theorem~\ref{extsubset_union}, in which we substitute extreme points by finite-dimensional polyhedral extreme subsets. Namely, we demonstrate that if $X$, $Y$ are Banach spaces, $F: B_X \to B_Y$ is a bijective non-expansive map and $S_Y$ is the union of all its finite-dimensional polyhedral extreme subsets, then $F$ is an isometry.

\vspace{2mm} Let us briefly explain the structure of the paper.
In Section~\ref{sec:uniformities} we extend the results about EC-spaces in totally bounded metric spaces to totally bounded uniform spaces, see Theorem~\ref{theo:compactuniform} and Lemma~\ref{lem:totallybounded}.

In Section~\ref{sec:previousbanach} we recollect some known results about bijective non-expansive maps between unit balls that we will need in the sequel.

The goal of Section~\ref{sec:Banach} is to demonstrate the main result. On the way we collect as much as possible information about preimages under a bijective non-expansive map of finite-dimensional faces of the unit ball. Using this information we obtain positive answers for the Problem~\ref{prob:ecbanach} for the case when $X$ is strictly convex (Theorem~\ref{plast-str-conv-dom}) and for the case when $S_Y$ is the union of all its finite-dimensional polyhedral extreme subsets (Theorem~\ref{extsubset_union}).

\vspace*{8mm}

We dedicate this paper to the memory of \emph{Bernardo Cascales}, who passed away in April, 2018. From the very beginning our activity related to EC-spaces was motivated by Bernardo's interest to the subject. It was his idea to search for a definition of EC-spaces that could be applicable to topological vector spaces and to uniform spaces. It was Bernardo who communicated to us the Ellis' result \cite{ellis58} and the way how this result was used by Isaac Namioka \cite{Namioka} in his elegant demonstration of EC-plasticity of compact metric spaces. We were planning to start a joint project, but\dots
\vspace{3 mm}

%%%%%%%%%%%%%%%%%%%%%%%%%%%%%%%%%%%%%%%%%%%%%%%%%%%%%%%%%%%%%%%%%%%%%%%%%%%%%%%%%%%%%%%%%%
%%%%%%%%%%%%%%%% Section Uniformities %%%%%%%%%%%%%%%%%%%%%
%%%%%%%%%%%%%%%%%%%%%%%%%%%%%%%%%%%%%%%%%%%%%%%%%%%%%%%%%%%%%%%%%%%%%%%%%%%%%%%%%%%%%%%%%%

\section{Non-expansive maps and uniformities}\label{sec:uniformities}

The aim of this section is to extend the results about EC-spaces in metric spaces to uniform spaces. We denote by $(E,\cU)$ a uniform space and by $\cB$ a basis of the uniformity. For $A\subset E$, $U,V \subset E\times E$, and $u,v,w \in E$ we denote:
$$
A^{-1} = \{(u,v) \colon (v, u) \in A	\}, \,\, \,\, (u,v)\circ (v,w)=(u,w);
 $$
 $$
 U\circ V=\{(u,v) \colon \text{ there is }w\in E\text{ such that }(u,w)\in U\text{ and } (w,v)\in V	\};
 $$
 $$
 U[A]=\{u\in E\text{ such that there is }v\in A\text{ with }(u,v)\in U\}.
 $$
The uniform space $(E,\cU)$ is called totally bounded if for every $U\in \cB$ there is a finite subset $\widetilde E \subset E$ such that $E = U[\widetilde E]$.

Let us recall the following definitions that were introduced in~\cite{ciric71} and extend the concepts of non-expansive, non-contractive and isometric maps to uniform spaces.

\begin{definition}
Let $(E,\cU)$ be a uniform space, $\cB$ a basis of the uniformity and $F:E\to E$ a map. We say that $F$ is \emph{non-contractive} for the basis $\cB$ if for every $V \in \cB$
 \begin{equation}\label{eq:finv2}
 (F(x),F(y))\in V \Rightarrow (x,y)\in V
\end{equation}
We say that $F$ is \emph{non-expansive} for the basis $\cB$ if for every $V \in \cB$
 $$
 (x,y)\in V \Rightarrow (F(x),F(y))\in V.
 $$
We say that $F$ is an \emph{isobasism} for the basis $\cB$ if for every $V \in \cB$
 $$
 (F(x),F(y))\in V \Leftrightarrow (x,y)\in V.
 $$
\end{definition}

For unexplained standard definitions and terminology we refer to \cite[Chapter 6]{Kelley}.

The next proposition will be used in the proof of Theorem~\ref{theo:compactuniform}.

\begin{propos}[Ellis \cite{ellis58}]\label{pro:Ellis}
Let $K$ be a compact space, let $S\subset C(K,K)$ be a semigroup for the composition, and let $\Sigma:=\overline{S}\subset K^K$. The following are equivalent:
\begin{enumerate}
 \item each member of $\Sigma$ is onto,
 \item each member of $\Sigma$ is one to one,
 \item $\Sigma$ is a group and $id:K\to K$ is the identity element of the group.
\end{enumerate}
\end{propos}

\begin{theo}\label{theo:compactuniform}
Let $K$ be a compact Hausdorff uniform space, $\cB$ a basis for the uniformity made of open sets in $K\times K$. If $F:K\to K$ is a non-contractive bijection for the basis $\cB$, then $F$ is an isobasism for the basis $\cB$.
\end{theo}

\begin{proof}
The demonstration follows the idea of Namioka's unpublished proof of EC-plasticity of compact metric spaces \cite{Namioka}, and is presented here with his kind permission.

Observe that since $F$ is non-contractive, then $F^{-1}$ is non-expansive and then
 $$
 (x,y)\in V \Rightarrow (F^{-1}(x),F^{-1}(y))\in V,
 $$
so $F^{-1}$ is a continuous bijection between compact spaces and then $F$ is a continuous function. Consider the semigroup
 $$
 S=\{F^n:n\in\N\}\subset C(K,K)
 $$
and let $G\in\Sigma=\overline{S}$ be the pointwise closure of $S$. Choose a net $(G_i)_{i\in I}$ in $S$ that converges to $G$. Let $x,y\in K$ and $V\in\cB$ be such that $(G(x),G(y))\in V$. There is $j\in I$ such that $(G_j(x),G_j(y))\in V$. Let $n\in \N$ be such that $G_j=F^n$. Then since $F$ is non-contractive we have that
 $$
 (F^n(x),F^n(y))\in V\Rightarrow (F^{n-1}(x),F^{n-1}(y))\in V\Rightarrow\cdots\Rightarrow (x,y)\in V.
 $$
We have proved that
 \begin{equation}\label{eq:ginV}
 (G(x),G(y))\in V \Rightarrow (x,y)\in V
 \end{equation}
for every $G\in\Sigma$, $x,y\in K$ and $V\in\cB$. Then we have that $G$ is one to one. By Proposition~\ref{pro:Ellis}, $\Sigma$ is a group so $F^{-1}\in\Sigma$ and then by~\eqref{eq:ginV} we have that
 $$
 (F^{-1}(x),F^{-1}(y))\in V \Rightarrow (x,y)\in V
 $$
for every $x,y\in K$ and $V\in\cB$, so $F^{-1}$ is non-contractive and then $F$ is an isobasism for the basis $\cB$.
\end{proof}

We know that every totally bounded metric space is an EC-space. The above theorem generalizes this result for uniformities when the space is compact. We can use some ideas of \cite{NaiPioWing} to get the following results for uniformities in totally bounded spaces.

\begin{lem}\label{lem:totallybounded}
Let $(E,\cU)$ be a totally bounded uniform space, $\cB$ a basis for the uniformity in $E\times E$ and $F:E\to E$ a non-contractive bijection for $\cB$.
Then $F$ satisfies that for every $V\in\cB$
 \begin{equation}\begin{split}\label{eq:xinV2}
 (x,y)\in V\Rightarrow &\text{ for each }W\in \cU\text{ there is }k\in\N\text{ such that }\\&(F^k(x),F^k(y))\in W\circ V \circ W.
 \end{split}\end{equation}
\end{lem}

\begin{proof}
Choose $x,y\in E$ and $V\in\cB$ with $(x,y)\in V$. Choose $W\in\cU$, $W'\in\cB$ a subset of $W$, $Z\in\cB$ such that $Z\circ Z\subset W'$ and $U\in \cB$ such that $U\subset Z\cap Z^{-1}$. Since $E$ is totally bounded there is a finite set $\widetilde E \subset E$ such that $E = U[\widetilde E]$. Then there is a infinite set $M\subset\N$ and $z_1,z_2\in \widetilde E$ such that $\{F^n(x):n\in M\}\subset U[z_1]$ and $\{F^n(y):n\in M\}\subset U[z_2]$. Pick $n,m\in M$ with $m>n$ and let $k=m-n$. Then
 $$
 (F^m(x),F^n(x))=(F^m(x),z_1)\circ (z_1,F^n(x))\in U\circ U^{-1}\subset Z\circ Z\subset W'.
 $$
Then by~\eqref{eq:finv2} we have that $(F^{k}(x),x)\in W'\subset W$. Analogously $(y,F^{k}(y))\in W$. Then
 $$
(F^k(x),F^k(y))=(F^k(x),x)\circ(x,y)\circ(y,F^k(y))\in W\circ V\circ W.
 $$
\end{proof}

% \textcolor{red}{Vladimir: It seems to me that in the lemma below the condition $U\circ V\in\cB$ if $U,V\in \cB$ can be omitted by the following reason. Let us say that $U \subset E\times E$ satisfies the $F$-contraction property, if
% $$
% (x,y)\in U \Rightarrow (F(x),F(y))\in U.
% $$
% Then it seems to me that for every two sets $U, V \subset E\times E$ that satisfy the $F$-contraction property, the corresponding $U \circ V$ also satisfies the $F$-contraction property. Indeed, let $(x, y) \in U \circ V$, then there is a $z \in X$ such that $(x, z) \in U$ and $(z, y) \in V$. In this case we have $(F(x), F(z)) \in U$ and $(F(z), F(y)) \in V$, so $(F(x), F(y)) \in U \circ V$. Am I missing something?}

\begin{corollary}\label{cor:expansiveclosed}
Let $(E,\cU)$ be a totally bounded uniform space, $\cB$ a basis for the uniformity and $F \colon E \to E$ a non-contractive bijection for $\cB$. Then $F$ satisfies that
 \begin{equation}
 (x,y)\in V\Rightarrow (F(x),F(y))\in \overline{V}
 \end{equation}
for every $V\in\cB$.
\end{corollary}

\begin{proof}
Choose $x,y\in E$ and $V\in\cB$ with $(x,y)\in V$. By Lemma~\ref{lem:totallybounded} we have that for each $W\in \cB$ there is $k\in \N$ such that $(F^k(x),F^k(y))\in W\circ V\circ W$. Then since $F^k$ is a bijection, we can choose $w,z\in E$ such that $(F^k(x),F^k(w))\in W$, $(F^k(w),F^k(z))\in V$ and $(F^k(z),F^k(y))\in W$. Since $F$ is a non-contractive map we have that $(F(x),F(w))\in W$, $(F(w),F(z))\in V$ and $(F(z),F(y))\in W$ so $(F(x),F(y))\in W\circ V\circ W$ and then
 $$
 (F(x),F(y))\in \bigcap_{W\in\cB}W\circ V\circ W=\overline{V}.
 $$
\end{proof}

Let $(E,d)$ be a metric space, if we denote $U_\epsilon=\{(x,y):d(x,y)<\epsilon\}$ then $\cB=\{U_\epsilon:\epsilon>0\}$ is a basis of the uniformity and $F:E\to E$ is non-expansive, non-contractive or an isometry for the metric $d$ if and only if $F$ is non-expansive, non-contractive or an isobasism for the basis of the uniformity $\cB$. Then Corollary~\ref{cor:expansiveclosed} implies the following result:

% Observe that in this case $U\circ V\in \cB$ if $U, V\in \cB$. \textcolor{red}{Vladimir: I don't see why the last claim is correct. The triangle inequality implies that $ U_\eps\circ U_\tau \subset U_{\eps + \tau}$ but it seems to me that the inverse inclusion is not always true. It is true for normed spaces or for convex subsets of normed spaces, but for example for $E = \{0, 3\} \subset \R$, we have $U_2 = \{(0,0), (3, 3)\}$, so $U_2 \circ U_2 = U_2 \neq U_4 = \{(0,0), (0,3), (3,0), (3, 3)\}$.
% Fortunately, if my previous ``red'' remark is correct, this difficulty disappears. }

\begin{corollary}[{\cite[Satz IV]{FH1936}}]
Let $(E,d)$ be a totally bounded metric space and $F:E\to E$ a bijective non-contractive (or non-expansive) map. Then $F$ is an isometry.
\end{corollary}

% \textcolor{red}{Again, if my first ``red'' remark is correct, the next corollary can be substituted by the following one, that looks more natural for me: let $X$ be a topological vector space, $A\subset X$ a totally bounded set and $\cB$ a basis of closed neighborhoods of 0. Let $F: A\to A$ be a bijection such that for every $x, y \in A$ and $V\in\cB$
% $$
% F(x)-F(y)\in V \Rightarrow x-y\in V.
% $$
% Then $f$ satisfies that for every $x, y \in A$ and $V\in\cB$
% $$
% x-y \in V \Rightarrow F(x)-F(y) \in V.
% $$
% }

\begin{corollary}
Let $X$ be a topological vector space, $A\subset X$ a totally bounded set and $\cB$ a basis of closed neighborhoods of 0. Let $F: A\to A$ be a bijection such that for every $x, y \in A$ and $V\in\cB$
 $$
 F(x)-F(y)\in V \Rightarrow x-y\in V.
 $$
Then $f$ satisfies that for every $x, y \in A$ and $V\in\cB$
 $$
 x-y \in V \Rightarrow F(x)-F(y) \in V.
 $$
\end{corollary}

\begin{proof}
This result follows from Corollary~\ref{cor:expansiveclosed} applied to the set $A$ and the basis for a uniformity $\{U_V:V\in\cB\}$ where $U_V=\{(x,y):x-y\in V\}$ for each $V\in\cB$.
\end{proof}

The following result is a reformulation of the last corollary:
\begin{corollary}
Let $X$ be a topological vector space, $A\subset X$ a totally bounded set and $\cB$ a basis of closed neighborhoods of 0. Let $F: A\to A$ be a bijection. If there is $x,y\in A$ and $V\in\cB$ such that
 $$
 x-y\in V\text{ and }F(x)-F(y)\notin V,
 $$
then there is $z,w\in A$ and $W\in \cB$ such that
 $$
 F(z)-F(w)\in W \text{ and } z-w\notin W.
 $$
\end{corollary}

%%%%%%%%%%%%%%%%%%%%%%%%%%%%%%%%%%%%%%%%%%%%%%%%%%%%%%%%%%%%%
%%%% Section Banach previous results %%%%%%%
%%%%%%%%%%%%%%%%%%%%%%%%%%%%%%%%%%%%%%%%%%%%%%%%%%%%%%%%%%%%%

\section{Notation and auxiliary statements for Banach spaces}\label{sec:previousbanach}

In this short section we fix the necessary notation and recollect some known results that we will need in the sequel. Below the letters $X$ and $Y$ always stand for real Banach spaces. We denote by $S_X$ and $B_X$ the unit sphere and the closed unit ball of $X$ respectively. For a convex set $A \subset X$ denote by $\ex(A)$ the set of extreme points of $A$; that is, $x \in \ex(A)$ if $x \in A$ and for every $y \in X\setminus\{0\}$ either $x + y \not\in A$ or $x - y \not\in A$. Recall that $X$ is called strictly convex if all elements of $S_X$ are extreme points of $B_X$, or in other words, $S_X$ does not contain non-trivial line segments. Strict convexity of $X$ is equivalent to the strict triangle inequality $\|x +y\| < \|x\| + \|y\|$ holding for all pairs of vectors $x, y \in X$ that do not have the same direction.
For subsets $A, B \subset X$ we use the standard notation $A+B = \{x + y\dopu x \in A, y \in B\}$ and $a A = \{ax \dopu x \in A\}$.

\begin{propos}[P.~Mankiewicz's \cite{mank}] \label{Mankiewicz}
If $A\subset X$ and $B \subset Y$ are convex subsets with non-empty interior, then every bijective isometry $F : A \to B$ can be extended to a bijective affine isometry $\widetilde F : X \to Y$.
\end{propos}
Taking into account that in the case of $A$, $B$ being the unit balls every isometry maps 0 to 0, this result implies that every bijective isometry $F : B_X \to B_Y$ is the restriction of a linear isometry from $X$ onto $Y$.

 \begin{propos} [Brower's invariance of domain principle \cite{Brouwer1912}] \label{Brower}
 Let $U$ be an open subset of $\R^n$ and $f : U \to \R^n$ be an injective continuous map, then $f(U)$ is open in $\R^n$.
 \end{propos}

\begin{propos}[{\cite[Proposition 4]{KZ}}] \label{prop-surject}
Let $X$ be a finite-dimensional normed space and $V$ be a subset of $B_X$ with the following two properties: $V$ is homeomorphic to $B_X$ and $V \supset S_X$. Then $V=B_X$.
\end{propos}

The remaining results of this section listed below appeared first in \cite{CKOW2016} for the particular case of $X = Y$. The generalizations to the case of two different spaces were made in \cite{Zav} and \cite{KZ2017}.

The following theorem appears in \cite[Theorem 2.1]{Zav} and it can be demonstrated repeating the proof of \cite[Theorem 2.3]{CKOW2016} almost word to word (see \cite[Theorem 2.3]{Zav2} for details).
\begin{theo} \label{wing-sc}
Let $F: B_X \to B_Y$ be a bijective non-expansive (briefly, a BnE) map. In the above notations the following hold.
\begin{enumerate}
\item $F(0) = 0$.
\item $F^{-1}(S_Y) \subset S_X$.
\item If $F(x)$ is an extreme point of $B_Y$, then $x$ is also an extreme point of $B_X$, $F(ax) = aF(x)$ for all $a \in [-1,1]$.

\end{enumerate}
Moreover, if $Y$ is strictly convex, then
\begin{enumerate}
\item[(i)] $F$ maps $S_X$ bijectively onto $S_Y$;
\item[(ii)] $F(ax) = a F(x)$ for all $x \in S_X$ and $a \in [-1, 1]$.
\end{enumerate}
\end{theo}

\begin{lem}[{\cite[Lemma 2.3]{Zav}}] \label{conv-smooth-prel}
 Let $F: B_X \to B_Y$ be a BnE map such that $F(S_X) = S_Y$. Let $V \subset S_X$ be the subset of all those $v \in S_X$ that $F(av) = a F(v)$ for all $a \in [-1,1]$. Denote $A = \{tx: x \in V, t \in [-1,1] \}$, then $F|_A$ is a bijective isometry between $A$ and $F(A)$.
\end{lem}

\begin{lem}[{\cite[Lemma 2.9]{KZ2017}}]\label{preim-of-sph}
 Let $F: B_X \to B_Y$ be a BnE map such that for every $v \in F^{-1}(S_Y)$ and every $t \in [-1,1]$ the condition
$F(tv) = t F(v)$ holds true. Then $F$ is an isometry.
\end{lem}

\begin{propos}[{\cite[Theorem 3.1]{Zav}}] \label{stric-conv-image}
 Let $F: B_X \to B_Y$ be a BnE map. If $Y$ is strictly convex, then $F$ is an isometry.
\end{propos}

Let us list some more definitions.

\begin{itemize}

\item An \emph{extreme subset} of a set $B\subset X$ is a subset $C \subset B$ with the property
$$
\forall_{y_1, y_2 \in B} \textit{ } \forall_{\alpha \in (0, 1) } \left(\alpha y_1 + (1-\alpha)y_2 \in C\right) \Longrightarrow (y_1, y_2 \in C).
$$

\item The \emph{generating subspace} of a convex set $C$ is $\spn (C - C)$.

\item The \emph{dimension of a convex set $C$} is the dimension of its generating subspace.

\item For a convex set $B\subset X$ we will say that a point $x \in B$ is \emph{$n$-extreme} if for any $(n+1)$-dimensional subspace $E \subset X$ and any $\varepsilon > 0$ there is an element $e \in S_E$, such that $x + \varepsilon e \notin B$.

\item For $n \in \N$ a point $x$ of the convex set $B$ is called \emph{sharp $n$-extreme in $B$} if it is $n$-extreme and is not $(n-1)$-extreme.

\end{itemize}

Remark, that in the definition we do not demand the convexity of extreme subsets. This is done in order to enjoy the following easy to verify property: the union of any collection of extreme subsets of $B$ is an extreme subset of $B$. Nevertheless, we mostly deal with convex sets and convex extreme subsets. Observe also that being $0$-extreme point and being extreme point of $B$ in the usual sense are equivalent. Every $n$-extreme point of $B$ is also $(n+1)$-extreme point of $B$. Every $n$-dimensional convex extreme subset $C$ of a convex set $B$ consists of $n$-extreme points of $B$ and contains a sharp $n$-extreme point. If $E$ is the generating subspace of the $n$-dimensional convex extreme subset $C \subset B$, then $x \in C$ is a sharp $n$-extreme point of $B$ if and only if $x$ belongs to the relative interior of $C$ in the affine subspace $x + E = C + E$. For a convex set $C\subset X$ with generating subspase $E$ by $\partial C$ we denote the relative boundary of $C$ in $C+E$.

 Evidently, in a normed space for collinear vectors $x, y$ looking in the same direction (codirected vectors) we have
\begin{equation} \label{Introd-quasi-collin}
\|x+y\|=\|x\|+\|y\|.
\end{equation}
In spaces that are not strictly convex the converse statement is not true, which motivates the following definition.
\begin{definition} \label{Def-Introd-quasi-collin}
Elements $x,y \in X$ are said to be {\it quasi-codirected}, if they satisfy (\ref{Introd-quasi-collin}).
\end{definition}
By the triangle inequality, in order to verify \eqref{Introd-quasi-collin} it is sufficient to check $\|x+y\| \ge \|x\|+\|y\|$. The next lemma is well-known, but this is the example of a fact which is much easier to demonstrate than to find out when and who observed it first \smiley
\begin{lem} \label{lemm-Introd-quasi-collin}
If $x,y \in X$ are quasi-codirected, then for every $a,b > 0$ the elements $ax$ and $by$ are quasi-codirected as well.
\end{lem}
\begin{proof}
By symmetry we may assume $a \ge b$. Then, $\|ax+by\| = \|a(x+y) - (a - b)y\| $ $\ge a\|x+y\| - (a - b)\|y\| = a\|x\|+b\|y\|$.
\end{proof}
Geometrically speaking $x,y \in S_X$ are quasi-codirected, if the whole segment
$$
[x, y] := \{tx + (1-t)y: t \in [0,1]\}
$$
lies on the unit sphere. If $C \subset S_X$ is convex, then every two elements of $C$ are quasi-codirected.

%%%%%%%%%%%%%%%%%%%%%%
%%%% Section Banach %%%%%%%
%%%%%%%%%%%%%%%%%%%%%%

\section{Non-expansive maps and finite-dimensional faces}\label{sec:Banach}

The aim of this section is, in the setting of Section~\ref{sec:previousbanach} and using some similar ideas, to obtain as much as possible information about preimages of finite-dimensional faces of the unit ball. The main result is Theorem~\ref{extsubset_union} that gives a positive answer for the Problem~\ref{prob:ecbanach} when $S_Y$ is the union of all its finite-dimensional polyhedral extreme subsets.

\vspace{2mm}

 Let us start with a very simple observation.

\begin{lem}\label{lem-quasi-codirected}
Let $X$, $Y$ be Banach spaces, $F: B_X \to B_Y$ be a BnE map, and $y_1, y_2 \in S_Y$ be quasi-codirected. Then,
\begin{enumerate}
\item $F^{-1}(y_1)$ is quasi-codirected with $-F^{-1}(-y_2)$, so
\item if $F^{-1}(-y_2) = - F^{-1}(y_2)$, then $F^{-1}(y_1)$ is quasi-codirected with $F^{-1}(y_2)$.
\item In particular if $y_2$ is an extreme point of $B_Y$, then $F^{-1}(y_1)$ is quasi-codirected with $F^{-1}(y_2)$.
\end{enumerate}
\end{lem}
\begin{proof}
\begin{align*}
\left\|F^{-1}(y_1) + \left(-F^{-1}(-y_2)\right) \right\| &= \left\|F^{-1}(y_1) - F^{-1}(-y_2) \right\| \\
&\ge \left\|y_1 - (-y_2) \right\| = \left\|y_1 + y_2 \right\| = 2.
\end{align*}
\end{proof}

The above lemma readily implies the following natural counterpart to Proposition \ref{stric-conv-image}.

\begin{theo}\label{plast-str-conv-dom}
Let $X$, $Y$ be Banach spaces, $X$ be strictly convex and $F: B_X \to B_Y$ be a BnE map. Then $F$ is an isometry.
\end{theo}

\begin{proof}
According to Proposition \ref{stric-conv-image} it is sufficient to demonstrate that $Y$ is strictly convex. Assume to the contrary that $S_Y$ contains a non-void segment $[y_0, y_1] := \{ty_1 + (1-t)y_0: t \in [0,1]\}$. Since $X$ is strictly convex, the only element of $S_X$ quasi-codirected with $F^{-1}(y_1)$ is $F^{-1}(y_1)$ itself. But, according to (1) of Lemma \ref{lem-quasi-codirected} all elements $-F^{-1}(-y_t)$, where $y_t := ty_1 + (1-t)y_0$, $t \in [0,1]$, are quasi-codirected with $F^{-1}(y_1)$. This contradiction completes the proof.
\end{proof}

Let $Y$ be a Banach space, $y_1, y_2 \in S_Y$ be quasi-codirected. Denote
\beq \label{set-D1y1y2}
\nonumber D_1(y_1, y_2) &:=& (y_1 + B_Y) \cap (- y_2 + B_Y) \\
&=& \left\{y \in Y \colon \|y_1 - y\| \le 1 \mathrm{\, and \, } \ \|y_2 + y\| \le 1 \right\}\\
\nonumber &=& \bigl\{y \in Y \colon \|y_1 - y\| = \|y_2 + y\| = 1 \bigr\}.
\eeq

% Remark, that in notations of \eqref{set-D1y} we have $D_1(y_1) = D_1(y_1, y_1)$.

Some evident properties of $D_1(y_1, y_2)$ are listed below without proof.
 \begin{lem}\label{lem-set-D1y1y2-ev-prop} Let $Y$ be a Banach space, $y_1, y_2 \in S_Y$ be quasi-codirected. Then
 \begin{itemize}
\item $D_1(y_1, y_2)$ is a convex closed subset of $Y$.
\item $0 \in D_1(y_1, y_2)$.
\item $t D_1(y_1, y_2) \subset D_1(y_1, y_2)$ for every $t \in [0,1]$.
\item $D_1(y_1, y_2) \subset 2 B_Y$, consequently
\item $\frac12 D_1(y_1, y_2) \subset D_1(y_1, y_2) \cap B_Y$.
\end{itemize}
\end{lem}

 \begin{lem}\label{lem-set-D1y1y2}
Let $Y$ be a Banach space, $y_1, y_2 \in S_Y$ be quasi-codirected, and $h \in Y$ be such that $y_1 \pm h \in S_Y$. Then
\begin{equation} \label{eq-D1(y1y2)1}
\Bigl\{\frac{1}{2}(y_1- y_2) \pm \frac{1}{2} h\Bigr\} \subset D_1(y_1, y_2).
\end{equation}
In particular, substituting $y_2 = y_1$ we obtain 
$$
\pm \frac{1}{2} h \in D_1(y_1, y_1).
$$
Substituting $h = 0$ we obtain
$$
\frac{1}{2}(y_1 - y_2) \in D_1(y_1, y_2),
$$
which implies that for all $t \in [0, 1/2]$
\begin{equation} \label{eq-D1(y1y2)3}
t (y_1 - y_2) \in D_1(y_1, y_2).
\end{equation}

\end{lem}
\begin{proof}
We have to verify two inequalities:
$$
\left\| \frac{1}{2}(y_1- y_2) \pm \frac{1}{2} h + y_2 \right\| \le 1 \mathrm{\, and \, } \ \left\| \frac{1}{2}(y_1- y_2) \pm \frac{1}{2} h - y_1 \right\| \le 1.
$$
Each of them reduces to the same inequality
$$
\left\| \frac{1}{2}(y_1 + y_2) \pm \frac{1}{2} h \right\| \le 1.
$$
Let us demonstrate this:
$\left\| (y_1 + y_2) \pm h \right\| = \left\| y_2 + (y_1 \pm h) \right\|
 \le \left\| y_2 \right\| + \left\| y_1 \pm h\right\| = 2$.
\end{proof}

 \begin{lem}\label{lem2-set-D1y1y2}
Let $Y$ be a Banach space, $C \subset S_Y$ be a convex extreme subset, and $E$ be the generating subspace of $C$. Then $D_1(y_1, y_2) \subset E$ for every $y_1, y_2 \in C$.
\end{lem}
\begin{proof}
Let $y \in D_1(y_1, y_2)$. Then, $y_1 - y, y_2 + y \in B_Y$ and
$$
\frac{1}{2}\bigl((y_1 - y) +( y_2 + y)\bigr) = \frac{1}{2}(y_1 + y_2) \in C.
$$
Consequently, by the definition of extreme subset, $y_2 + y \in C$, so $y = ( y_2 + y) - y_2 \in C-C \subset E$.
\end{proof}

 \begin{lem}\label{lemD1inimage}
Let $X$, $Y$ be Banach spaces, $F: B_X \to B_Y$ be a BnE map, $y_1, y_2 \in S_Y$ be quasi-codirected, $x_1 = F^{-1}(y_1) \in S_X$, $x_2 = -F^{-1}(-y_2) \in S_X$. Then
$$
F\left(D_1(x_1, x_2) \cap B_X\right) \subset D_1(y_1, y_2) \cap B_Y.
$$
In particular, $F\left(\frac12 D_1(x_1, x_2) \right) \subset D_1(y_1, y_2) \cap B_Y$.
\end{lem}
\begin{proof}
According to (1) of Lemma \ref{lem-quasi-codirected}, $x_1$ and $x_2$ are quasi-codirected, so the set $D_1(x_1, x_2) $ is well-defined. Consider arbitrary $x \in D_1(x_1, x_2) \cap B_X$. We have $\|x_1 - x\| \le 1$ and $\|(-x_2) - x\| = \|x_2 + x\| \le 1$, so $\|F(x_1) - F(x)\| \le 1$ and $\|F(-x_2) - F(x)\| \le 1$. In other words,
$\|y_1 - F(x)\| \le 1$ and $\|(-y_2) - F(x)\| = \|y_2 + F(x)\| \le 1$, which means that $ F(x) \in D_1(y_1, y_2)$.
\end{proof}

\begin{lem}\label{n-extpoint}
Let $X$, $Y$ be Banach spaces, $F: B_X \to B_Y$ be a BnE map, $n\in \N$, and $C \subset S_Y$ be an $n$-dimensional convex extreme subset. Then for every $y_1 \in C$ its preimage $x_1 = F^{-1}(y_1) \in S_X$ is an $n$-extreme point of $B_X$.
\end{lem}
\begin{proof}
Denote $x_2 = -F^{-1}(-y_1) \in S_X$. Assume that $x_1$ is not $n$-extreme point of $B_X$. Then, according to the definition, there exist an $(n+1)$-dimensional subspace $E \subset X$ and an $\varepsilon > 0$ such that $x_1 + \varepsilon B_E \subset S_X$. According to Lemma \ref{lem-set-D1y1y2}
$$
\frac{1}{2}(x_1- x_2) + \varepsilon B_E \subset D_1(x_1, x_2).
$$
The above inclusion implies that $\frac12 D_1(x_1, x_2)$ contains an $(n+1)$-dimensional ball. Then Lemma \ref{lemD1inimage} implies that $D_1(y_1, y_1)$ contains a homeomorphic copy of $(n+1)$-dimensional ball, which is impossible by Lemma \ref{lem2-set-D1y1y2}.
\end{proof}
Note, that under conditions of the previous lemma $x$ may be also $m$-extreme point for some $m < n$. Now we are coming to the most important and at the same time most difficult result of the paper.

\begin{theo} \label{theo-preim-extr-set}
Let $X$, $Y$ be Banach spaces, $F: B_X \to B_Y$ be a BnE map, then for every $n\in \N$ the preimage of any $n$-dimensional convex polyhedral extreme subset $C \subset S_Y$ is an $n$-dimensional convex polyhedral extreme subset of $S_X$. Moreover, the equality $- F^{-1}(-C) = F^{-1}(C)$ holds true.
\end{theo}
\begin{proof}
We will use the induction in $n$. The initial case of $n = 0$ (i.e., of extreme points) is covered by the assertion (3) of Theorem \ref{wing-sc}. Let us assume that the theorem is demonstrated for extreme subsets of dimension smaller than $n$, and let us demonstrate it for a given $n$-dimensional polyhedral extreme subset $C \subset S_Y$. Denote $E$ the generating subspace of $C$, $\dim E = n$. The boundary $\partial C$ of polyhedron $C$ consists of finite union of its convex $(n-1)$-dimensional polyhedral extreme subsets, so, by the inductive hypothesis, $A:= F^{-1}(\partial C)$ also consists of finite union of some convex $(n-1)$-dimensional polyhedral extreme subsets of $S_X$. Consequently, $A$ is an extreme subset of $B_X$. Also, $A$ is compact, and $F|_A$ performs a homeomorphism between $A$ and $F(A) = \partial C$. Let $y_1 \in C \setminus \partial C$ be an arbitrary point. Denote $x_1 = F^{-1}(y_1)$. Since $y_1$ is quasi-codirected with every point $y_2 \in \partial C$, $x_1$ is quasi-codirected with the corresponding $x_2 = -F^{-1}(-y_2) \in S_X$. By the inclusion \eqref{eq-D1(y1y2)3} and Lemmas \ref{lemD1inimage} and \ref{lem2-set-D1y1y2}
$$
F\left(t (x_1 - x_2)\right) \in F\left(\frac12 D_1(x_1, x_2) \right) \subset D_1(y_1, y_2) \subset E
$$
for all $t \in \left[0, \frac{1}{4}\right]$.
By the inductive hypothesis, when $y_2$ runs through $\partial C$ the corresponding $x_2$ runs through the whole $A$. So, denoting $\widetilde A = \left[0, \frac{1}{4}\right] (x_1 - A) = \left\{t (x_1 - x_2) \colon t \in \left[0, \frac{1}{4}\right], x_2 \in A \right\}$ we obtain
\begin{equation} \label{eq-tildeA}
F\bigl(\widetilde A\bigr) \subset E.
\end{equation}
Let us demonstrate that the segments $\left(0, \frac{1}{4}\right] (x_1 - x_2)$ with different $x_2 \in A$ are pairwise disjoint. We will argue by contradiction. Let two segments of the form $\left(0, \frac{1}{4}\right] (x_1 - \widehat{x}_2)$, $\left(0, \frac{1}{4}\right] (x_1 - \widetilde{x}_2)$ with $\widehat{x}_2,\widetilde{x}_2 \in A$, $\widehat{x}_2 \neq \widetilde{x}_2$ intersect at some point $y$. Then the corresponding closed segments $\left[0, \frac{1}{4}\right] (x_1 - \widehat{x}_2)$, $\left[0, \frac{1}{4}\right] (x_1 - \widetilde{x}_2)$ intersect in two points ($0$ and $y$), so either they coincide or one segment contains the other one. That is, $(x_1 - \widehat{x}_2)$ and $(x_1 - \widetilde{x}_2)$ are codirected. There are two cases:$(x_1 - \widehat{x}_2) = \lambda(x_1 - \widetilde{x}_2)$ or $(x_1 - \widetilde{x}_2) = \lambda(x_1 - \widehat{x}_2)$ with some $0<\lambda<1$. We will discuss the first one, the second one is analogous. We get $\widehat{x}_2=\lambda\widetilde{x}_2+(1-\lambda)x_1$, so these three points are on the same segment and $\widehat{x}_2$ lies between $x_1$ and $\widetilde{x}_2$. Since $A$ is extreme subset, we get $x_1\in A$, which contradicts the fact $y_1\notin \partial C$.

The set $(x_1 - A)$ is homeomorphic to the unit sphere of $\R^n$.
Let us show, that $\widetilde A$ is homeomorphic to the unit ball of $\R^n$, with 0 mapped to 0. Let $S^n$ and $B^n$ denote the unit sphere and the unit ball of $\R^n$ respectively, and $h \colon S^n \to (x_1 - A)$ be a homeomorphism.
 One may define the mapping $H\colon B^n \to \widetilde A$ as
 \begin{equation*}
H(x)=
\begin{cases}
 0, & \mbox{when } x=0 \\
 \frac{1}{4}\|x\| h\left(\frac{x}{\|x\|}\right), & \mbox{when } x\in B^n\backslash\{0\}.
\end{cases}
 \end{equation*}
 Obviously, this mapping is bijective and continuous at 0. We are going to show that $H$ is continuous at all points.
 Let us consider some sequence $\{x_n\}_{n=1}^{\infty}$ in $B^n$ converging to an $x \in B^n \setminus \{0\}$, that is
 $$
 \lim_{n\to\infty}x_n = x \neq 0.
 $$
 Then
 \begin{align*}
 \lim_{n\to\infty}H(x_n)&=\lim_{n\to\infty}\frac{1}{4}\|x_n\|h\left(\frac{{x}_n}{\|x_n\|}\right)=\frac{1}{4}\lim_{n\to\infty}\|x_n\| \lim_{n\to\infty}h\left(\frac{{x}_n}{\|x_n\|}\right)\\
&=\frac{1}{4}\|x\| h\left(\lim_{n\to\infty}\frac{x_n}{\|x_n\|}\right)=\frac{1}{4}\|x\|h\left(\frac{{x}}{\|x\|}\right) = H(x).
\end{align*}
So, $H$ is a bijective continuous map from compact $B^n$ to Hausdorf space, thus $H$ is a homeomorphism.

 Consequently, $F\bigl(\widetilde A\bigr) \subset E$ is homeomorphic to the unit ball of $\R^n$, with 0 being a relative (in $E$) interior point of $F\bigl(\widetilde A\bigr)$.

Consider now any point $\widetilde y_2 \in C \setminus \partial C$, $\widetilde y_2 \neq y_1$, such that the corresponding $\widetilde x_2 = -F^{-1}\bigl(-\widetilde y_2\bigr)$ is not equal to $x_1$. By the same reason as before, the segment $F\left(\left[0, \frac{1}{4}\right](x_1 - \widetilde x_2)\right) \subset D_1(y_1, \widetilde y_2) \subset E$. The set $F\left(\left[0, \frac{1}{4}\right] (x_1 - \widetilde x_2)\right)$ is a continuous curve in $E$ connecting $F(\frac{1}{4}(x_1 - \widetilde x_2))$ with 0, which is an interior point of $F\bigl(\widetilde A\bigr)$. So there is a $t_0 \in \left(0, \frac{1}{4}\right]$ such that $F\bigl(t_0 (x_1 - \widetilde x_2)\bigr) \in F\bigl(\widetilde A\bigr)$, that is $t_0 (x_1 - \widetilde x_2) \in \widetilde A$. This means that for some $t_1 \in \left(0, \frac{1}{4}\right]$ and some $x_2 \in A$ we have $t_0 (x_1 - \widetilde x_2) = t_1 (x_1 - x_2)$. In other words, there is an $\alpha > 0$ such that
\begin{equation} \label{eq-tot1}
x_1 - \widetilde x_2 = \alpha (x_1 - x_2).
\end{equation}
 Let us demonstrate that $\alpha < 1$. Indeed, if $\alpha \ge 1$, the above formula would give the representation
 $$
 x_2 = \left(1 - \frac{1}{\alpha}\right)x_1 + \frac{1}{\alpha}\widetilde x_2
 $$
 of $x_2 \in A$ as a convex combination of $x_1, \widetilde x_2 \in S_X \setminus A$, which contradicts the fact that $A$ is extreme in $S_X$.

 Since $\alpha < 1$, the formula \eqref{eq-tot1} gives the representation
 $$
 \widetilde x_2 = (1 - \alpha) x_1 + \alpha x_2
 $$
 of $\widetilde x_2$ as a convex combination of $x_1$ and some $x_2 \in A$.

 If we consider the BnE mapping $G: B_X \to B_Y$ defined as $G(x) = -F(-x)$, all the above reasoning is applicable for $G$ as well, because by the inductive hypothesis $G^{-1}(\partial C) = F^{-1}(\partial C) = A$. Since $\widetilde x_2 = G^{-1}\bigl(\widetilde y_2\bigr)$ and $x_1 = - G^{-1}(- y_1)$ the roles of these elements for $G$ interchange, and we deduce that also $x_1$ is a convex combination of $\widetilde x_2$ and some $x_3 \in A$. So, we obtain the following properties of sets $F^{-1}(C \setminus \partial C)$ and $G^{-1}(C \setminus \partial C)$:

\smallskip
\textbf{Properties.}

\noindent (i) For every $u \in F^{-1}(C \setminus \partial C)$
$$
G^{-1}(C \setminus \partial C) \subset \left\{t x + (1-t)u \colon t \in \left[0, 1\right], x \in A \right\}.
$$

\noindent (ii)
For every $v \in G^{-1}(C \setminus \partial C)$
$$
F^{-1}(C \setminus \partial C) \subset \left\{t x + (1-t)v \colon t \in \left[0, 1\right], x \in A \right\}.
$$

\noindent (iii) For every $u \in F^{-1}(C \setminus \partial C)$ and every $v \in G^{-1}(C \setminus \partial C)$, $u \neq v$ there are (unique) elements $w, z \in A$ such that $[u, v] \subset [w, z]$.
\smallskip

Properties (i) and (ii) imply that $F^{-1}(C)$ and $G^{-1}(C)$ lie in some finite-dimensional subspace of $X$. Since both these sets are bounded and closed, they are compacts. Continuous mappings $F$ and $G$ map corresponding compacts $F^{-1}(C)$ and $G^{-1}(C)$ to $C$ bijectively, so both $F^{-1}(C)$ and $G^{-1}(C)$ are homeomorphic to $C$, i.e. homeomorphic to the unit ball of $\R^n$. Since the set $\left\{t x + (1-t)u \colon t \in \left[0, 1\right], x \in A \right\}$ for a fixed $u$ is also homeomorphic to the unit ball of $\R^n$ and $A$ corresponds to the unit sphere and belongs to both $\left\{t x + (1-t)u \colon t \in \left[0, 1\right], x \in A \right\}$ and $G^{-1}(C)$, the inclusion (i) and Proposition \ref{prop-surject} imply that

\smallskip
\noindent (i)' for every $u \in F^{-1}(C \setminus \partial C)$
$$
G^{-1}(C) = \left\{t x + (1-t)u \colon t \in \left[0, 1\right], x \in A \right\},
$$
and by the same reason

\noindent (ii)'
For every $v \in G^{-1}(C \setminus \partial C)$
$$
F^{-1}(C) = \left\{t x + (1-t)v \colon t \in \left[0, 1\right], x \in A \right\}.
$$

In particular, from (i)' it follows that every $u \in F^{-1}(C \setminus \partial C)$ belongs to $G^{-1}(C)$, so $F^{-1}(C) \subset G^{-1}(C)$, and (ii)' implies the inverse inclusion $G^{-1}(C) \subset F^{-1}(C)$, so
$$
G^{-1}(C) = F^{-1}(C).
$$

Coming back to the already used inclusion \eqref{eq-D1(y1y2)3} and Lemmas \ref{lemD1inimage} and \ref{lem2-set-D1y1y2} we obtain that for all $x_1, x_2 \in F^{-1}(C)$
$$
F\left(\frac{1}{4}(x_1 - x_2)\right) \in E,
$$
in other words
\begin{equation}\label{*}
 F\left(\frac{1}{4}( F^{-1}(C) - F^{-1}(C))\right) \subset E.
\end{equation}

Recall, that by the inductive hypothesis, $A= F^{-1}(\partial C)$ consists of finite union of some convex $(n-1)$-dimensional polyhedral extreme subsets $\widetilde W_i$, $i = 1, \ldots, N$ which are preimages of corresponding parts of $\partial C$. Let us fix some $v\in F^{-1}(C \setminus \partial C)$.
Denote
$$
W_i = \left\{t x + (1-t)v \colon t \in \left[0, 1\right], x \in \widetilde W_i \right\}.
$$
These $W_i$ are $n$-dimensional convex polyhedrons, and, according to (ii)',
$$
 F^{-1}(C) = \bigcup_{i=1}^N W_i.
$$
We state that all polyhedrons $W_i$ (and also their union $ F^{-1}(C)$) are situated in one and the same $n$-dimensional affine subspace $\widetilde E$.

To this end, consider the generating subspaces $Z_i = \spn (W_i - W_i)$ of $W_i$ and let us demonstrate that all of $Z_i$ are equal one to another, i.e. all of them are equal to some $n$-dimensional linear subspace $Z$. Then $\widetilde E = v + Z$ will be the $n$-dimensional affine subspace $\widetilde E$ we are looking for.

Let us argue ``ad absurdum''. Assume that $Z_i \neq Z_j$ for some $i \neq j$. Then $Z_i + Z_j$ has dimension strictly greater than $n$, and
$$
\dim(W_i-W_j) = \dim(\spn((W_i-W_j) - (W_i-W_j))) = \dim(Z_i + Z_j) > n.
$$
Taking into account that $W_i-W_j \subset (F^{-1}(C) - F^{-1}(C))$ the dimension of $F^{-1}(C) - F^{-1}(C)$ is strictly greater than $n$, which makes the inclusion \eqref{*} impossible.

It remains to demonstrate that $ F^{-1}(C)$ is convex and is an extreme subset. For the convexity let us show that $ F^{-1}(C) = B_X \cap \widetilde E$.
We have already known, that $ F^{-1}(C)\subset B_X \cap \widetilde E$. Let us show the inverse inclusion. Again we will argue by contradiction. Suppose there is a point $z\in (B_X \cap \widetilde E)\setminus F^{-1}(C)$. We may fix some $v\in F^{-1}(C \setminus \partial C)$ and consider the segment $[z,v]$. As we already remarked, $F^{-1}(C)$ is homeomorphic to $C$ and hence to $B^n$, that is, $v$ lies in the relative interior of $F^{-1}(C)$ in $\widetilde E$. So, the segment $[z,v]$ must intersect $A = F^{-1}(\partial C)$ in some point. In other words, there is $\lambda\in(0,1)$ such that $\lambda z+ (1-\lambda)v\in A$, which contradicts the fact, that $A$ is an extreme subset in $B_X$.

\end{proof}

\begin{theo}\label{conv(0,C)}
 Let $X$, $Y$ be Banach spaces, $F: B_X \to B_Y$ be a BnE map, then for every $n$-dimensional convex polyhedral extreme subset $C \subset S_Y$ the following equality holds true:
 $F(\conv(0,F^{-1}(C)))= \conv(0,C)$.
\end{theo}
\begin{proof}
We will carry out the proof by induction in $n$. For $n = 0$ (i.e., when C is extreme point) the required equality may be obtained from the assertion (3) of Theorem \ref{wing-sc}. Suppose our theorem is proved for all extreme subsets of dimension smaller than $n$, and let us show the same for a given $n$-dimensional polyhedral extreme subset $C \subset S_Y$. Consider $x\in F^{-1}(C\setminus \partial C)$ and $\alpha \in (0,1)$.
Since $F$ is non-expansive we have
\begin{equation}\label{**}
\|F(\alpha x)\|\leq \|\alpha x\|, \textrm{ and } \|F(x)-F(\alpha x)\|\leq \|x-\alpha x\|.
\end{equation}
Also
\begin{align}\label{eq:3star}
\nonumber 1=\|F(x)\| &\leq \|F(\alpha x)\|+\|F(x)-F(\alpha x)\| \\
&\leq \|\alpha x\|+\|x-\alpha x\|=1.
\end{align}
That is why
$$
\|F(\alpha x)\|+\|F(x)-F(\alpha x)\| = 1.
$$
So one may write $F(x)$ as a convex combination
\begin{align*}
 F(x) = \|F(\alpha x)\|\frac{F(\alpha x)}{\|F(\alpha x)\|}+\|F(x)-F(\alpha x)\|\frac{F(x)-F(\alpha x)}{\|F(x)-F(\alpha x)\|}.
\end{align*}
Since $F(x)\in C$ and $C$ is extreme subset in $B_X$ we get $\frac{F(\alpha x)}{\|F(\alpha x)\|}\in C$ and $\frac{F(x)-F(\alpha x)}{\|F(x)-F(\alpha x)\|}\in C$. So, $F(\alpha x) = \|F(\alpha x)\|\frac{F(\alpha x)}{\|F(\alpha x)\|}\in \conv(\frac{F(\alpha x)}{\|F(\alpha x)\|},0) \subset \conv(0,C)$ and thus $F(\conv(0,F^{-1}(C)))\subset \conv(0,C)$.
 By the inductive hypothesis $F(\conv(0,A))= \conv(0,\partial C)$ and $\partial \conv(0,C) \subset F(\conv(0,F^{-1}(C)))$. Besides, $\conv(0,F^{-1}(C))$ is homeomorphic to $B^{n+1}$ and $\partial \conv(0,C)$ is homeomorphic to $S^{n+1}$. In this way Proposition \ref{prop-surject} implies the statement of the theorem.
\end{proof}

\begin{lem}\label{normeq}
Let $X$, $Y$ be Banach spaces, $F: B_X \to B_Y$ be a BnE map, then
$\|F(\alpha x)\| =\|\alpha x\|=\alpha$ for all $x\in F^{-1}(S_Y)$, $\alpha \in [0,1]$.
\end{lem}
\begin{proof}
Since $F$ is non-expansive, we may use inequalities \eqref{**} and \eqref{eq:3star}.
The inequality \eqref{eq:3star} implies
$$
\|F(\alpha x)\|+\|F(x)-F(\alpha x)\| =\|\alpha x\|+\|x-\alpha x\|,
$$
and application of \eqref{**} concludes the proof.
\end{proof}

\begin{theo}\label{extsubset_union}
 Let $X$, $Y$ be Banach spaces, $F: B_X \to B_Y$ be a BnE map and $S_Y$ be the union of all its finite-dimensional polyhedral extreme subsets. Then $F$ is an isometry.
\end{theo}

\begin{proof}
Let us first show, that $F( S_X) = S_Y$.
Since
\begin{equation}\label{eqS_Y}
 S_Y=\bigcup_{i\in I}^{}C_i,
\end{equation}
 where $C_i$ are finite-dimensional polyhedral extreme subsets of $S_Y$ and $I$ is some index set, one may deduce
 $$B_Y=\bigcup_{i\in I}\conv(0,C_i).$$
 Due to bijectivity of $F$, theorem \ref{conv(0,C)} implies
 $$B_X=\bigcup_{i\in I}\conv(0,F^{-1}(C_i)).$$
 Consequently, there is no other norm-one points in $B_X$ except for points from $F^{-1}(C_i)$, and we get
$$
S_X=\bigcup_{i\in I}F^{-1}(C_i)=F^{-1}(S_Y).
$$
To prove that $F$ is an isometry we will use lemmas \ref{preim-of-sph} and \ref{conv-smooth-prel}. We are going to show for the set $V$ from lemma \ref{conv-smooth-prel} that
\begin{equation}\label{eq:cinV1}
F^{-1}(C)\subset V
\end{equation}
 for every $n$-dimensional polyhedral extreme subset $C$ of $S_Y$. To do that, we will use induction by dimension. For $0$-dimensional sets, i.e. extreme points, the statement we need follows from item (3) of theorem \ref{wing-sc}.
 Now suppose that the inclusion is proved for all $(n-1)$-dimensional polyhedral extreme subsets and let us prove it for dimension $n$. Consider some $n$-dimensional extreme subset $C$ in $S_Y$. For every pair $x,y\in F^{-1}(C)$ there are $u,v\in F^{-1}(\partial C)$ such that $x=\lambda u + (1-\lambda)v $ and $y=\mu u + (1-\mu)v, \lambda,\mu \in(0,1)$. Without loss of generality one may account $\lambda>\mu$. Since $\partial C$ consists of $(n-1)$-dimensional polyhedral extreme subsets, the inductive hypothesis and lemma \ref{conv-smooth-prel} give that $\|u-v\|=\|F(u)-F(v)\|$. Since $F$ is non-expansive,
 \begin{align*}
 \|u-v\|&=\|F(u)-F(v)\|\leq \|F(u)-F(x)\|+\|F(x)-F(y)\|\\
 &+\|F(y)-F(v)\|\leq \|u-x\|+\|x-y\|+\|y-v\|\\
 &= (1-\lambda)\|u-v\|+(\lambda-\mu)\|u-v\|+\mu\|u-v\|=\|u-v\|.
 \end{align*}
 So we get $\|F(u)-F(x)\|=\|u-x\|$, $\|F(y)-F(v)\|=\|y-v\|$, $\|F(x)-F(y)\|= \|x-y\|$. Thus, $F$ is bijective isometry between $F^{-1}(C)$ and $C$ and Proposition \ref{Mankiewicz} implies that $F$ is affine on $F^{-1}(C)$. Lemma \ref{normeq} together with Theorem \ref{conv(0,C)} give the equality $F(\alpha F^{-1}(C)) = \alpha C$ for $\alpha \in [0, 1]$, and application of the ``moreover" part of theorem \ref{theo-preim-extr-set} extends this to $\alpha\in [-1,1]$. The same way as before,
 the inductive hypothesis and lemma \ref{conv-smooth-prel} imply that $F$ is bijective isometry between $\alpha F^{-1}(C)$ and $\alpha C$, so $F$ is affine on $\alpha F^{-1}(C)$. We are going to show that $F(\alpha x)=\alpha F(x)$ for all $x \in F^{-1}(C)$, $\alpha \in [-1,1]$. Every $x \in F^{-1}(C)$ is of the form $x=\lambda u+(1-\lambda)v$, where $u,v\in F^{-1}(\partial C)$ and $\lambda\in (0,1)$. We obtain
 \[
 F(\alpha x) = F( \lambda \alpha u+(1-\lambda)\alpha v) = \lambda F(\alpha u)+(1-\lambda)F(\alpha v),
 \]
because $F$ is affine on $\alpha F^{-1}(C)$. By the inductive hypothesis $F(\alpha u) =\alpha F(u)$, $F(\alpha v) = \alpha F(v)$, so
 \[
 F(\alpha x) = \lambda \alpha F(u) + (1-\lambda) \alpha F(v) = \alpha (\lambda F(u) + (1-\lambda)F(v)).
 \]
It remains to use the fact that $F$ is affine on $F^{-1}(C)$ to conclude that
 \[
 F(\alpha x) = \alpha F(\lambda u + (1-\lambda)v) = \alpha F(x).
 \]
 So, the required inclusion \eqref{eq:cinV1} is demonstrated.
 At last, \eqref{eqS_Y} and the written above imply that for every $v \in F^{-1}(S_Y)$ and every $t \in [-1,1]$ $F(tv) = t F(v)$. So, the application of lemma \ref{preim-of-sph} completes the proof of the theorem.
 \end{proof}

\bibliographystyle{amsplain}

\end{document}